\documentclass[11pt]{article}
\usepackage{amsthm}
\usepackage{amsmath}
\usepackage{amssymb}
\usepackage{enumerate}
\usepackage{epsfig}
\usepackage{paralist}

\usepackage[dvipsnames,usenames]{color}

\newcounter{contador}
\newtheorem{propo}[contador]{Proposition}
\newtheorem{teo}[contador]{Theorem}
\newtheorem{lem}[contador]{Lemma}

\newtheorem{defi}[contador]{Definition}
\newtheorem{corol}[contador]{Corollary}

\newcommand{\sign}{{\rm sign}}

\newcommand{\R}{{\mathbb R}}

\newcommand{\N}{{\mathbb N}}

\newcommand{\Z}{{\mathbb Z}}

\title{Basin of attraction of  triangular maps with applications
\footnote{The authors are supported by
Ministry of Economy and Competitiveness grants MTM2008-03437 (first
and second authors); DPI2011-25822 (third author). GSD-UAB and
CoDALab Groups are supported by the Government of Catalonia through
the SGR program. }}
\author{Anna Cima$^{(1)}$, Armengol Gasull$^{(1)}$ and V\'{\i}ctor
Ma\~{n}osa$^{(2)}$
\\*[.1truecm]
{\small \textsl{$^{(1)}$ Dept. de Matem\`{a}tiques, Facultat de Ci\`{e}ncies,}}
\\*[-.25truecm] {\small \textsl{Universitat Aut\`{o}noma de Barcelona,}}
\\*[-.25truecm] {\small \textsl{08193 Bellaterra, Barcelona, Spain}}
\\*[-.25truecm] {\small \textsl{cima@mat.uab.cat, gasull@mat.uab.cat}}
\\*[-.25truecm]
\\*[-.25truecm] {\small \textsl{$^{(2)}$ Dept. de Matem\`{a}tica Aplicada III (MA3),}}
\\*[-.25truecm] {\small \textsl{Control, Dynamics and Applications Group (CoDALab)}}
\\*[-.25truecm] {\small \textsl{Universitat Polit\`{e}cnica de Catalunya (UPC)}}
\\*[-.25truecm] {\small \textsl{Colom 1, 08222 Terrassa, Spain}}
\\*[-.25truecm] {\small \textsl{victor.manosa@upc.edu}}}

\begin{document}
\maketitle

\begin{abstract}
We consider planar triangular maps
$x_{n+1}=f_0(u_n)+f_1(u_n)x_n,u_{n+1}=\phi(u_n)$. These  maps
preserve the fibration of the plane given by
$\mathcal{F}=\{\phi(u)=c, c\in \mathrm{Image}(\phi)\}$.  We assume
that there exists an invariant attracting fiber $\{u=u_*\}$ for the
dynamical system generated by $\phi$ and we study the limit dynamics
of those points in the basin of attraction of this invariant fiber,
assuming that either it contains a global attractor, or it is filled
by fixed or $2$-periodic points. Finally, we apply our results to a
variety of examples, from particular cases of triangular systems to
some planar quasi-homogeneous maps, and some multiplicative and
additive difference equations, as well.
\end{abstract}

\bigskip

\noindent {\sl 2000 Mathematics Subject Classification:}
\texttt{39A10, 39A11, 39A20}

\bigskip


\newpage
\section{Introduction}

In this paper we consider triangular systems of the form
\begin{equation}\label{problema1}
\left\{
  \begin{array}{ll}
    x_{n+1} & = f_0(u_n)+f_1(u_n)\,x_n, \\
    u_{n+1} & =\phi(u_n),
  \end{array}
\right.
\end{equation}
where $\{x_n\}$ and $\{u_n\}$ are real sequences, and $f_0$, $f_1$
and $\phi$ are continuous functions.

Observe that  system (\ref{problema1}) preserves the fibration of
the plane given by $\mathcal{F}=\{\phi(u)=c, c\in
\mathrm{Image}(\phi)\}$, that is, it sends fibers of $\mathcal{F}$
to other ones. We will assume that there exists a point $u=u_*$
which is a local stable attractor of the subsystem
$u_{n+1}=\phi(u_n)$. In this case, we say that system
(\ref{problema1}) has a \emph{local attractive fiber} $\{u=u^*\}$.
Our objective is to know if the asymptotic dynamics of the orbits
corresponding to points in the \emph{basin of attraction} of the
\emph{limit fiber} $\{u=u_*\}$ is characterized by the dynamics on
this fiber, the \emph{limit dynamics}. In all the cases, we will
assume that the limit dynamics is very simple, that is: either (A)
the fiber $\{u=u_*\}$ contains a global attractor, see Proposition
\ref{propoprob1attractor}; (B) the fiber is filled by fixed points,
Theorem \ref{propoprob1identitatv2}(a); or  (C) it is filled by
$2$-periodic orbits, Theorem \ref{propoprob1identitatv2}(b). Observe
that  there is no need to consider the case in which there is a
global repellor on the fiber $\{u=u_*\}$, since in this situation it
is clear that any orbit with initial conditions in the basin of
attraction of this fiber is unbounded (otherwise there should be
accumulation points in the fiber).

The paper is structured as follows. In Section \ref{main} we present
the main results (Proposition \ref{propoprob1attractor}, Theorem
\ref{propoprob1identitatv2} and Corollary \ref{corolaridiff})
together with some motivating examples. The proofs are given in
Section \ref{proofs}.

Section \ref{Aplications1} is devoted to show some applications of
the main results. For instance, in Section \ref{quasi1} we study the
limit dynamics of some linear quasi-homogeneous maps. The example
considered there shows that, for this class of maps, the shape of
the basins of attraction of the origin can have a certain level of
complexity.

In Section \ref{secondordermultiplicative} we apply our results to
study the global dynamics of difference equations of multiplicative
type $ x_{n+2}=x_ng(x_nx_{n+1})$, with $g$  being a $\mathcal{C}^1$
function. In particular, we reobtain the results in \cite{BL}
concerning the  difference equations $
x_{n+2}={x_n}/{(a+bx_nx_{n+1})},\,a,b\in\R.$ The more general
 recurrences  $x_{n+k}=x_ng(x_n\ldots x_{n+k-1})$ are
studied in Section \ref{higher}. We also apply the results to study
the additive difference equations
$x_{n+2}=-bx_{n+1}+g(x_{n+1}+bx_{n})$ in Section~\ref{Aplications3}.

  Finally, we also present several different type of
recurrences for which the theory developed in  this work applies in
order to obtain a full characterization of their dynamics.

\section{Motivating examples and main results}\label{main}

We started motivated by the following   clarifying example, inspired
in the ones developed for continuous systems in \cite{CGM97} (a
particular case is also considered in \cite{IB}, see also \cite{XX}
and \cite{X}), which shows that there are systems of type
\begin{equation}\label{triang} \left\{
  \begin{array}{ll}
    x_{n+1} & =\sum\limits_{\ell\geq 0} f_\ell(u_n) x_n^\ell, \\
    u_{n+1} & =\lambda u_n+o(u_n),
  \end{array}
 \right.\end{equation}
with $|\lambda|<1$, such that  $\{u=0\}$ is a global limit fiber,
having a global attractor for the restricted dynamics on $\{u=0\}$,
 but having unbounded orbits   $(x_n,u_n)$, with $u_n\to 0$, in their the global dynamics.

\medskip

\noindent \textsl{Example A: Hyperbolic globally attracting fiber
with a restricted global attractor but having unbounded solutions.}
Consider the system
\begin{equation}\label{cascuriosgeneral}
\left\{
  \begin{array}{ll}
    x_{n+1} & = \mu\, x_n+ a\, u_n^j x_n^\ell, \\
    u_{n+1} & =\lambda\, u_n,
  \end{array}
\right.
\end{equation}
with $|\lambda|<1$, $|\mu|\leq 1$ and $a\in\R$. Of course, this
system preserves the fibration $\{u=c;\, c\in \R\}$, and it has a
\emph{global} attracting limit fiber $\{u=0\}$. On this fiber the
dynamics is the following: the origin is attractive if $|\mu|<1$;
the fiber is a continuum of fixed points if $\mu=1$; and the fiber
is a continuum of $2$--periodic orbits if $\mu=-1$. However, as
shown in Proposition \ref{propocascurios},  if $j(\ell -1)>0$ then
there are initial conditions giving rise to unbounded solutions, and
therefore the global dynamics is not characterized by the dynamics
  on the global limit fiber.

\begin{propo}\label{propocascurios}
 The curve
$$
\Gamma=\left\{u^jx^{\ell-1}=\frac{\lambda^\alpha-\mu}{a},\mbox{
where } \alpha=-\frac{j}{\ell-1}\right\}
$$ is invariant for system (\ref{cascuriosgeneral}). Moreover, for any initial condition $(x_0,u_0)\in\Gamma$
the associated orbit is given by $x_n=\left(\lambda^\alpha\right)^n
x_0$, $u_n=\lambda^n u_0$. If $j(\ell -1)>0$, then  any initial
condition $(x_0,u_0)\in\Gamma$ gives rise to
 unbounded solutions, with $u_n\to 0$.
\end{propo}

\begin{proof}
Imposing that system (\ref{cascuriosgeneral}) has  solutions of the
form $x_n=\left(\lambda^\alpha\right)^n x_0$, $u_n=\lambda^n u_0$,
we easily get
$$
\lambda^\alpha \,x_0= \mu\, x_0+ a\,\lambda^{nj+\alpha n(\ell-1)}\,
u_0^j x_0^\ell.
$$

Observe that if  $ nj+\alpha n(\ell-1)=0$  (that is when
$\alpha=-j/(\ell-1)$)  we obtain that $\Gamma$ is an invariant
curve. Moreover,  if $j(\ell -1)>0$, then $|\lambda^\alpha|>1$, and
the orbits on $\Gamma$ are unbounded solutions.
\end{proof}

Example A shows that, in general, for  systems of type
(\ref{triang}) with terms of degree greater or equal than $2$ in
$x_n$, we cannot expect that  the global
 dynamics in the basin of attraction of the limit
fiber $\{u=0\}$ is characterized by the dynamics on the fiber.
 However, it remains to explore the case of
systems of type (\ref{problema1}) with affine terms in $x_n$. For
these systems, the possible limit dynamics that we consider are the
cases (A), (B) and (C) mentioned in the introduction.

\medskip

The first result concerns the case (A), when the the limit fiber
contains a global attractor. This case is characterized by the fact
that   $|f_1(u_*)|<1$, and the attractor is
$x_*=f_0(u_*)/(1-f_1(u_*))$;

\begin{propo}\label{propoprob1attractor}
Consider the system (\ref{problema1}) with $f_0$ and $f_1$
continuous
 and $|f_1(u_*)|<1$. Suppose that $u=u_*$ is an
attractive point of $u_{n+1}=\phi(u_n)$. Then, for all initial
condition $(x_0,u_0)$ with $u_0$ in the basin of attraction of
$u=u_*$, we have $\lim\limits_{n\to
\infty}(x_n,u_n)=(f_0(u_*)/(1-f_1(u_*)),u_*)$.
\end{propo}

We want to point out that the above one is not a local result. In
fact, the convergence is guaranteed for all $(x_0,u_0)$ such that
$u_0$ is in the basin of attraction of $u=u_*.$
\medskip

It is not difficult to see that in the case (B), when the limit
fiber $\{u=u_*\}$ is filled by fixed points, there are systems such
that the orbits corresponding to initial conditions in the basin of
attraction of the limit fiber tend to some of the fixed points on
the limit fiber, which depend on the initial conditions. Indeed:

\medskip

\noindent \textsl{Example B: Hyperbolic attracting fiber and fast
enough convergence to the limit dynamics.} Consider the system given
by
$$ \left\{
  \begin{array}{ll}
    x_{n+1} & = \left(1+a\,u_n\right)\, x_n, \\
    u_{n+1} & =\lambda\, u_n,
  \end{array}
\right.
$$
with $|\lambda|<1$, and $a\in\R$. Then
$$
x_{n+1}=\left(\prod\limits_{k=0}^{n} \left(1+a\lambda^k
u_0\right)\right)\, x_0.
$$
Observe that the infinite product $P(u_0)=\prod_{k=0}^{\infty}
\left(1+a\lambda^k u_0\right)$ is convergent since, as
$|\lambda|<1$, we have
$$ S(u_0)=\ln(P(u_0))=\sum\limits_{k=0}^{\infty} \ln(1+a\lambda^k
u_0)\sim \sum\limits_{k=0}^{\infty} a\lambda^k
u_0=\frac{au_0}{1-\lambda}<\infty.
$$ So, for each initial condition $(x_0,u_0)$ the associated orbit
converges to the fixed point $(P(u_0)x_0,0)$.

\medskip

The next example shows, however, that the above situation does not
occur when the convergence of $f_1(u_n)$ to $1$ is too slow.

\medskip

\noindent \textsl{Example C: Hyperbolic attracting fiber but slow
convergence to the limit dynamics.} Consider the system
$$
\left\{
  \begin{array}{ll}
    x_{n+1} & = f_1(u_n)\, x_n, \\
    u_{n+1} & =\lambda\, u_n,
  \end{array}
\right.
$$
with $|\lambda|<1$, and $$f_1(u)=\left\{
                                  \begin{array}{ll}
                                    \displaystyle{1-\frac{1}{\ln |u|}} & \hbox{ for } u\neq 0, \\
                                    1 & \hbox{ for } u=0.
                                  \end{array}
                                \right.
$$
Again, an straightforward computation gives
$$
x_{n+1}=\left(\prod\limits_{k=0}^{n}
\displaystyle{\left(1-\frac{1}{k\ln|\lambda|+\ln|u_0|}\right)}\right)\,
x_0.
$$
Using that $\sum\limits_{k=0}^\infty
\ln\left(1-\displaystyle{\frac{1}{k\ln|\lambda|+\ln|u_0|}}\right)\sim
\sum\limits_{k=0}^\infty
\displaystyle{\frac{1}{k\ln|\lambda|+\ln|u_0|}}$, which is a
divergent series, we get  that the infinite product
$\prod\limits_{k=0}^{\infty}
\left(1-\displaystyle{\frac{1}{k\ln|\lambda|+\ln|u_0|}}\right)$ is
also divergent. Hence each initial condition $(x_0,u_0)$ gives rise
to an unbounded sequence $\{x_n\}$ and, therefore,  the dynamics on
the basin of attraction of the limit fiber $\{u=0\}$ is not
characterized by the dynamics on the fiber.

\medskip

The above example leads us to introduce the following definition. As
we will see in Example E below, it gives an optimal characterization
of the speed of convergence of the terms $f_0(u_n)$ and $f_1(u_n)$
to guarantee the convergence of the orbits in the basin of
attraction of the limit fiber.

\begin{defi}[Fast
enough convergence to the limit dynamics property]\label{fecld} If
$u_*$ is a stable attractor of $u_{n+1}=\phi(u_n)$, we say that the
attracting fiber $\{u=u_*\}$ of  system (\ref{problema1}) with
$|f_1(u_*)|=1$, has \emph{the fast enough convergence to the limit
dynamics property} if there exists $\varepsilon>0$  such that if
$\{u_n\}$ is a solution of $u_{n+1}=\phi(u_n)$ satisfying
$|u_{n_0}-u_*| \leq \varepsilon$ for some $n_0$,  then there exists
$p_n$ such that $|u_{n}-u_*|\leq p_n\leq \varepsilon$ for all $n\geq
n_0$, $\lim\limits_{n\to\infty} p_n=0$, and also there exist two
functions $V,W:[0,\varepsilon)\rightarrow \R^+$ such that for all
$|u-u_*|<\varepsilon$ it is satisfied:
\begin{enumerate}
  \item[$H_1$:] $|f_0(u)-f_0(u_*)|\leq W(|u-u_*|)\leq \overline{W}\in\mathbb{R}$.
  \item[$H_2$:] $|f_1(u)-f_1(u_*)|\leq V(|u-u_*|)\leq \overline{V}\in\mathbb{R}$.
  \item[$H_3$:] $V(\nu)$ and $W(\nu)$ are non decreasing for $0\leq \nu<\varepsilon$;
  \item[$H_4$:]  $S_W=\sum_{j=0}^\infty W(p_j)<\infty$ and $S_V=\sum_{j=0}^\infty V(p_j)<\infty$.
\end{enumerate}
\end{defi}

\begin{defi}[]   Let $u=u_*$
be a stable attractive fixed point of $\phi(u)$.  We say that
$\phi(u)$ is \emph{locally contractive} at $u=u_*$ if there exists
an open neighborhood $V$ of $u_*$ such that any $u\in
V\setminus\{u_*\}$
\begin{equation}\label{locatr1}
|\phi(u)-u_*|\,<\,|u-u_*|.
\end{equation}
\end{defi}

Prior to state the next results, we   state the following result
about local contractivity.

\begin{propo}\label{expans}
Let $\phi:\mathcal{U}\to \mathcal{U}$ be a continuous function and
let $u_*$ be a stable attractive fixed point of $\phi(u)$.
\begin{enumerate}
\item [(a)] If $\phi$ is an orientation preserving function, then $\phi(u)$
is locally contractive at $u=u_*$.
\item [(b)] There exist orientation reversing functions $\phi(u)$ which are not locally contractive at $u=u_*$.
\item [(c)] If $\phi\in\mathcal{C}^1(\mathcal{U})$ and
$u=u_*$ is an hyperbolic fixed point of $\phi(u)$  then $\phi(u)$ is
locally contractive at $u=u_*$.
\end{enumerate}
\end{propo}

Now we are ready to   present our result about the cases (B) and
(C). Observe that if $f_1(u_*)=1$ and $f_0(u_*)=0$, then the limit
fiber is filled by fixed points (if $f_0(u_*)\neq 0$ then there are
unbounded solutions). When $f_1(u_*)=-1$, the fiber
 is filled by $2$-periodic orbits    and there is a fixed point given by $x_*=f_0(u_*)/2$.

\begin{teo}\label{propoprob1identitatv2}
Consider system (\ref{problema1}) where $f_0$ and $f_1$ are
continuous functions,  and $|f_1(u_*)|=1$. Suppose that $\{u=u_*\}$
is an attracting fiber satisfying the fast enough convergence to the
limit dynamics property, and consider initial conditions $(x_0,u_0)$
with $u_0$  in the basin of attraction of $u=u_*$ for the recurrence
$u_{n+1}=\phi(u_n).$ Then
\begin{enumerate}
  \item[(a)] If $f_1(u_*)=1$ and $f_0(u_*)=0$, then there exists
$\ell(x_0,u_0)\in \mathbb{R}$ such that $\lim\limits_{n\to
\infty}(x_n,u_n)=(\ell(x_0,u_0),u_*).$

  \item[(b)] If $f_1(u_*)=-1$ and additionally $\phi$ is locally
  contractive at $u=u_*$, then
$\lim\limits_{n\to \infty}(x_{2n},u_{2n})=(\ell(x_0,u_0),u_*)$ and
$\lim\limits_{n\to
\infty}(x_{2n+1},u_{2n+1})=(f_0(u_*)-\ell(x_0,u_0),u_*).$
\end{enumerate}
\end{teo}

Notice that the above result must not be interpreted in the sense
that the limit $\ell(x_0,u_0)$ is different for each initial
condition $(x_0,u_0)$.

\medskip

\noindent \textsl{Example D: Hyperbolic attracting fiber and fast
enough convergence to the limit dynamics via Theorem
\ref{propoprob1identitatv2}.} Consider the systems given by $$
\left\{
  \begin{array}{ll}
    x_{n+1} & = \left(a+b\,|u_n|^\alpha\right)\, x_n, \\
    u_{n+1} & =\lambda\, u_n,
  \end{array}
\right.
$$
with $|\lambda|<1$, $\alpha>0$ , $a\in\{-1,1\}$ and $b\in\R$.

Fixing  $\varepsilon$ and taking $|u_0|<\varepsilon$;  setting
$p_n=|\lambda|^n |u_0|$,
 $W(u)\equiv 0$, $\overline{W}=0$, $V(u)=B|u|^\alpha$ with $B=|b|$,  $\overline{V}=B\varepsilon^\alpha$ we
have that hypothesis $H_1$ and $H_3$ are trivially fulfilled. With
respect hypothesis $H_2$ we have
$$|f_1(u)-f_1(0)|=B|u|^\alpha=V(|u|)\leq B
\varepsilon^\alpha=\overline{V}.$$ On the other hand $S_W=0$ and
  $$
S_V=\sum_{j=0}^\infty V(p_j)=\sum_{j=0}^\infty
V(|\lambda|^j|u_0|)=\sum_{j=0}^\infty B |\lambda|^{j \alpha
}|u_0|^{\alpha}=\frac{B|u_0|^\alpha}{1-|\lambda|^\alpha},
  $$ so $H_4$ is also fulfilled. From Theorem
  \ref{propoprob1identitatv2}, for each initial condition
  $(x_0,u_0)$ there exists $n_0$ such that $|u_{n_0}|<\varepsilon$,
  and also there exists $\ell(x_0,u_0)\in \mathbb{R}$ such that: if
  $a=1$ then
$\lim\limits_{n\to \infty}(x_n,u_n)=(\ell(x_0,u_0),0)$, and if
$a=-1$, then
  $\lim\limits_{n\to
\infty}(x_{2n},u_{2n})=(\ell(x_0,u_0),0)$ and  $\lim\limits_{n\to
\infty}(x_{2n+1},u_{2n+1})=(-\ell(x_0,u_0),0)$.

\medskip

The next example shows that Theorem \ref{propoprob1identitatv2} is
optimal when $u=u_*$ is not a hyperbolic attractor of
$u_{n+1}=\phi(u_n)$.

\medskip

\noindent \textsl{Example E: Non-hyperbolic attracting fiber.
Optimal characterization of fast enough and slow convergence to the
limit dynamics.} Consider the systems given by
\begin{equation}\label{case} \left\{
  \begin{array}{ll}
    x_{n+1} & = \left(1+b\,|u_n|^\alpha\right)\, x_n, \\
    u_{n+1} & = |u_n|-a |u_n|^k,
  \end{array}
\right.
\end{equation}
with $a>0$, $b\neq 0$, and $k>1$.

Take $0<|u_0|<\varepsilon$ small enough and such that $u_0$ is in
the basin of attraction of $0$. Now, we claim that if $\alpha> k-1,$
then there exists $\ell(x_0,u_0)\in \mathbb{R}$ such that
$\lim\limits_{n\to \infty}(x_n,u_n)=(\ell(x_0,u_0),0)$; and if
$\alpha \leq k-1$ then $\{x_n\}$ is a divergent sequence if $b>0$,
and $\lim\limits_{n\to\infty}x_n=0$ if $b<0$. To prove this
  claim, we will use the following result that we
learned from S.~Stevi\'c \cite{Stev}.   Its proof is attributed to
E. Jacobsthal, see also \cite[Problem 174, page 217]{PZ}:

\begin{teo}[\textbf{E. Jacobsthal}]\label{stev}
Let $f:(0,\bar{x})\longrightarrow (0,\bar{x})$ with $\bar{x}>0$, be
a continuous function such that $0<f(x)<x$ for every
$x\in(0,\bar{x})$ and $f(x)=x-ax^k+bx^{k+p}+o(x^{k+p})$, where
$k>1$, $p,a$ and $b$ are positive. Let $x_0\in(0,\bar{x})$ and
$x_n=f(x_{n-1})$ $n\geq 1$. Then
$$
x_n\sim \frac{1}{\left((k-1)\,a\,n\right)^{\frac{1}{k-1}}}
$$
\end{teo}

Taking logarithms in Equation  (\ref{case}) we have
 $\ln|x_{n+1}|=\sum\limits_{j=0}^n\ln|1+b\,|u_j|^\alpha|+\ln|x_0|$.
By Theorem~\ref{stev} we have,
$$
 \ln|x_{n+1}|=\sum\limits_{j=1}^\infty
\ln|1+b\,|u_j|^\alpha|+\ln|x_0|\sim b\, \sum\limits_{j=1}^\infty
|u_j|^\alpha\sim
\frac{b}{\left((k-1)a\right)^{\frac{\alpha}{k-1}}}\,
\sum\limits_{j=1}^\infty \frac{1}{j^{\frac{\alpha}{k-1}}}.
$$
Hence if $\alpha>k-1$, then $\{\ln |x_n|\}$ is a convergent
sequence, and if $\alpha\leq k-1$ then  $\lim\limits_{n\to\infty}
\ln |x_n|=\sign(b)\infty$, and the claim is proved.

At this point we will see that the criterium given by Theorem
\ref{propoprob1identitatv2} is optimal.   By fixing $\varepsilon>0$
small enough, taking $|u_0|<\varepsilon$, and using Theorem
\ref{stev}, we have that for the sequence $\{u_n\}$ defined by
(\ref{case}) and for all $\delta>0$, there exists $n_0$ and $A>0$
such that if $n>n_0$ then
$$
|u_n|\leq \frac{A}{n^{\frac{1}{k-1}}}\leq
\frac{1}{n^{\frac{1}{k-1}-\delta}}.
$$
Setting
$$p_n=\frac{1}{n^{\frac{1}{k-1}-\delta}},$$
 $W(u)\equiv 0$, $\overline{W}=0$, $V(u)=|u|^\alpha$, and $\overline{V}=\varepsilon^\alpha$ we
have that for any $\delta>0$, the hypothesis $H_1$--$H_3$ are
trivially fulfilled, and with respect $H_4$ we have that  $S_W=0$
and
  $$
S_V=\sum_{j=1}^\infty V(p_j)=\sum_{j=1}^\infty
\frac{1}{j^{\frac{\alpha}{k-1}-\delta}}.
 $$

Observe that $S_V$ is convergent if and only if
$\frac{\alpha}{k-1}-\delta>1$ for all $\delta>0$ or, in other words,
if and only if
$$
\alpha>(k-1)(1+\delta) \mbox{ for all } \delta>0.
$$
Hence, Theorem \ref{propoprob1identitatv2} guarantees convergence of
the sequence $\{(x_n,u_n)\}$ if $\alpha>k-1$ for all $b\neq 0$,
which is the optimal value.

\medskip

The next result shows that if system (\ref{problema1}) is a
differentiable one, and $u=u_*$ is a hyperbolic attractor of
$u_{n+1}=\phi(u_n)$, then the hypotheses of Theorem
\ref{propoprob1identitatv2} are fulfilled. Furthermore, each point
in the attracting fiber is the limit of an orbit of the basin of
attraction of the fiber.

\begin{corol}\label{corolaridiff}
Consider system (\ref{problema1}) where
$f_0,f_1,\phi\in\mathcal{C}^1(\mathcal{U})$ where $\cal{U}$ is a
neighborhood of $u=u_*$, and $|f_1(u_*)|=1$. Suppose that $u=u_*$ is
a hyperbolic attractor of $\phi$.  Then, for any $u_0$ in the basin
of attraction of $u=u_*$ there exists $\ell(x_0,u_0)\in\R$ such
that:
\begin{enumerate}
  \item[(a)] If $f_0(u_*)=0$ and $f_1(u_*)=1$, then  $\lim\limits_{n\to
\infty}(x_n,u_n)=(\ell(x_0,u_0),u_*)$.
  \item[(b)] If $f_1(u_*)=-1$, then  $\lim\limits_{n\to
\infty}(x_{2n},u_{2n})=(\ell(x_0,u_0),u_*)$ and  $\lim\limits_{n\to
\infty}(x_{2n+1},u_{2n+1})=(f_0(0)-\ell(x_0,u_0),u_*)$.
\end{enumerate}
Furthermore, for any point $(x_*,u_*)\in\{u=u_*\}$  there exists an
initial condition $(x_0,u_0)$ in the basin of attraction of the
limit fiber such that $\lim\limits_{n\to\infty}
(x_n,u_n)=(x_*,u_*)$.
\end{corol}

\section{Proofs of the main results}\label{proofs}

In this section we prove the main results of the paper,
Proposition~\ref{propoprob1attractor},
Theorem~\ref{propoprob1identitatv2} and
Corollary~\ref{corolaridiff}.

\begin{proof}[Proof of Proposition \ref{propoprob1attractor}] First observe that, from the continuity of $f_1$, there exists
$\varepsilon$ such that  for all $u$ such that
$|u-u_*|<\varepsilon$, we have $|f_1(u)|< N$ with $N<1$. Consider
$(x_0,u_0)$ such that $u_0$ is in the basin of attraction of
$u=u_*$. From the hypotheses, we can assume that there exists $n_0$
such that for all $n\geq n_0$, we have $|u_n-u_*|<\varepsilon$
hence, since $f_0$ is continuous,
$|f_0(u_n)|<M=:\sup_{|u-u_*|<\varepsilon}|f_0(u)|$. Thus we have
$|x_{n+n_0+1}| \leq M+N |x_{n+n_0}|$. Applying this last inequality
we obtain
$$\begin{array}{rl}
|x_{n+n_0+1}| &\leq M+N |x_{n+n_0}|\leq
M+N\left(M+N|x_{n+n_0-1}|\right)\leq
\ldots\leq M\left(\sum\limits_{i=0}^n N^i\right)+N^n|x_{n_0}| \\
& \leq \frac{M}{N-1}+|x_{n_0}|,
  \end{array}
$$
and therefore the sequence is bounded. As a consequence
$i:=\liminf\limits_{n\to \infty}\in\R$ and $s:=\limsup\limits_{n\to
\infty}\in\R$. Hence, there are subsequences
$\{(x_{n_k},u_{n_k})\}\to (i,u_*)$ and $\{(x_{n_j},u_{n_j})\}\to
(s,u_*)$. For these subsequences,  from equation (\ref{problema1}),
and using the continuity of $f_0$ and $f_1$, we have
$$
s=f_0(u_*)+f_1(u_*)s\, \mbox{ and }  i=f_0(u_*)+f_1(u_*)i,
$$
hence we have that $i=s=f_0(u_*)/(1-f_1(u_*))$.
\end{proof}

\begin{proof}[Proof of Proposition \ref{expans}] To prove $(a)$, since $u=u_*$ is a stable attractor,
we can take  $\varepsilon>0$ such that for all $u$ with
$|u-u_*|<\varepsilon,$ $\lim\limits_{n\to\infty}\phi^n(u)=u_*.$ In
particular, $u_*$ is the only fixed point in the neighborhood
$(u_*-\varepsilon,u_*+\varepsilon).$ Observe that since $\phi$
preserves orientation, if $u>u_*$ then $\phi(u)>u_*.$

Assume that statement $(a)$ is not true. Given a sequence
$\{\varepsilon_n\}\to 0$, then for each $n$ with
$\varepsilon_n<\varepsilon$ we can find a point $u_n$ such that
\begin{equation}\label{estrella}
|u_n-u_*|<\varepsilon_n\quad\text{and}\quad |\phi(u_n)-u_*|\ge
|u_n-u_*|.
\end{equation} Taking, if necessary, a subsequence, we can assume that
$u_n>u_*$ for all $n$ (the opposite assumption can be treated in a
similar manner),  and also that the subsequence is monotone
decreasing. Notice that, since $u_n>u_*$ and $\phi$ preserves
orientation, also $\phi(u_n)>u_*$ and the second inequality in
(\ref{estrella}) reads as
$$\phi(u_n)\ge u_n\,\,\text{for all}\,\,n\in\N.$$
On the other hand, we can consider the sequence of iterates by
$\phi$ of one point $w_0>u_*$ such that $w_0-u_*<\varepsilon$,
obtaining $w_n:=\phi^n(w_0)>u_*$ for all $n\in\N$ and that
$\lim\limits_{n\to\infty}w_n=u_*.$ That is, we can choose a
subsequence $w_{k_n}$ with the property that $w_{k_n}<u_n$ and, as
before, being monotone decreasing. Considering the continuous
function   $\phi-Id$ we have that for all $n\in\N,$
$$(\phi-Id)(u_n)=\phi(u_n)-u_n\ge 0 \mbox{ and } (\phi-Id)(w_{k_n})=\phi(w_{k_n})-w_{k_n}<
0.$$ Hence, for all $n\in\N$ it exists $t_n$ such that
$$w_{k_n}<t_n<u_n\,\,,\,\,\phi(t_n)=t_n.$$
It implies that $\lim\limits_{n\to\infty}t_n=u_*$ and $u_*$ is not
isolated as a fixed point of $\phi$. A contradiction with our
assumptions.

In order to see $(b)$ consider the function $\phi(u)=-u-u^2.$ This
map has two fixed points $u=0$ and $u=-2.$ Since $\phi'(-2)=3>1,$
the point $u=-2$ is a repellor of the function $\phi(u).$ The point
$u=0$ is not an hyperbolic fixed point because $\phi'(0)=-1,$ but it
is easy to see that the interval $\left[-2,\frac{1}{4}\right]$ is
invariant under the action of $\phi$ and that for all
$u\in\left(-2,\frac{1}{4}\right)$ and
$\lim\limits_{n\to\infty}\phi^n(u)= 0.$ Then, $u_*=0$ is an
attracting fixed point of $\phi(u).$

On the other hand, if $u>0$ then $\phi(u)$ satisfies
$|\phi(u)|=|u+u^2|>|u|,$ hence condition (\ref{locatr1}) is not
satisfied.

Statement (c) is a simple consequence of the mean value theorem.


\end{proof}

\begin{proof}[Proof of Theorem \ref{propoprob1identitatv2}]
In order to prove statement $(a)$ consider $(x_0,u_0)$ such that
$u_0$ is in the basin of attraction of $u=u_*$. First we start
proving that the sequence $x_n$ is bounded. From the hypothesis, if
$\varepsilon>0$ is small enough,  there exists $n_0$ such that for
all $n\geq n_0$, $|u_n-u_*|<\varepsilon$.  Hence, to simplify the
notation, in the following we will assume that the point $(x_0,u_0)$
is such that  $|u_0-u_*|<\varepsilon$. Furthermore,
\begin{enumerate}
  \item[(1)] From the fact that $W(\nu)$ is non decreasing for $0\leq \nu<\varepsilon$, if $n$ is large enough
  $|f_0(u_{n}) -f_0(u_*)|=|f_0(u_{n})|\leq W(|u_{n}-u_*|)\leq W( p_n)\leq \overline{W}.$
  \item[(2)] Analogously  $|f_1(u_{n})|\leq |f_1(u_*)|+V(|u_{n}-u_*|)\leq
  1+V(p_n)$.
\end{enumerate}
In summary, we have $|x_{n+1} -f_0(u_*)| \leq
W(p_n)+\left(1+V(p_n)\right) |x_{n}|$. Applying this last inequality
we obtain
$$\begin{array}{rl} |x_{n+1} -f_0(u_*)| &\leq W(p_n)+\left(1+V(p_n)\right) |x_{n}|\\
& \leq
  W(p_n)+\left(1+V(p_n)\right)W(p_{n-1})+\left(1+ V(p_n)\right)\left(1+V(p_{n-1})\right)|x_{n-1}|\\
& \leq \ldots \leq W(p_n)+ \sum\limits_{i=0}^{n-1}
\left[W(p_i)\,\prod\limits_{j=i+1}^{n}
\left(1+V(p_j)\right)\right]+\prod\limits_{j=0}^{n}
\left(1+V(p_j)\right)|x_{0}|.
  \end{array}
$$
Observe that from hypothesis $H_2$ we have $V(0)=0$, and by
hypothesis $H_3$, for all $j=0,\ldots n$ we have that $1+V(p_j)\geq
1$. So
$$
\prod\limits_{j=i+1}^{n} 1+V(p_j)\leq \prod\limits_{j=0}^{n}
1+V(p_j)\leq \prod\limits_{j=0}^{\infty} 1+V(p_j)=:P.
$$
Observe that $P<+\infty$, because
$$
S=\ln(P)=\sum\limits_{j=0}^{\infty}\ln\left( 1+V(p_j)\right)\sim
\sum\limits_{j=0}^{\infty} V(p_j)=S_V<+\infty
$$
Hence, regarding that if  $n$ is large enough then  $W( p_n)\leq
  \overline{W}$, we have
$$
|x_{n+1} -f_0(u_*)|\leq \overline{W} +P
\left(\sum\limits_{i=0}^{n-1} W(p_i)\right)+P\,|x_{0}|\leq
\overline{W} +P \, S_W+P\,|x_{0}|,
$$
so the sequence $\{x_n\}$ is bounded, and therefore also the
sequence $\{(x_n,u_n)\}$ is bounded.

Now we are going to see  that $\{x_n\}$ is a Cauchy sequence and
therefore it has a limit $\ell(x_0,u_0)$. First observe that if $R$
is a bound of $\{|x_n|\}$ then
$$
|x_{n+1}-x_{n}|\leq |f_0(u_{n})|+|f_1(u_{n})-1|\, |x_{n}|\leq
W(p_n)+V(p_n)R.
$$
Therefore, since the series $S_W$ and $S_V$ are convergent, then for
all $\varepsilon>0$ there exists $N>0$ such that for all $n, m>N$
$$\begin{array}{rl}
  |x_{n+m}-x_{n}|&\leq |x_{n+m}-x_{n+m-1}|+\ldots+|x_{n+1}-x_{n}| \\
    & \leq
    \sum\limits_{j=n}^{n+m-1} W(p_j)+R\sum\limits_{j=n}^{n+m-1} V(p_j)\leq
    \varepsilon,
  \end{array}
$$
and therefore $\{x_n\}$ is a Cauchy sequence, which completes the
proof of statement (a).

In order to see $(b)$, observe that
$$
\begin{array}{rl}
    x_{n+2}&=f_0(u_{n+1})+f_1(u_{n+1})\left(f_0(u_{n})+f_1(u_{n})
x_{n}\right)\\&=:F_0(u_{n})+F_1(u_{n})x_{n},
  \end{array}
$$
where $F_0(u)=f_0(\phi(u))+f_1(\phi(u))f_0(u)$ and
$F_1(u)=f_1(\phi(u))f_1(u)$. After renaming $y_{n}=x_{2n}$,
 $v_{n}=u_{2n},$ $v_*=u_*$ and $\varphi=\phi\circ\phi$ we get the
system
\begin{equation}\label{dospassosy}
 \left\{
  \begin{array}{ll}
     y_{n+1}&=F_0(v_{n})+F_1(v_{n})y_{n}, \\
    v_{n+1} & = \varphi(v_{n}),
  \end{array}
\right.
\end{equation}
which is a system of type (\ref{problema1}). Notice that when we
take $(x_0,u_0)$ (resp. $(x_1,u_1)$) as initial condition and we
apply (\ref{dospassosy}) repeatedly we get $(x_{2n},u_{2n})$ (resp.
$(x_{2n+1},u_{2n+1})$).

We are going to see that system (\ref{dospassosy}) satisfies the
hypothesis $H_1$--$H_4$, and since $F_0(v_*)=0$ and $F_1(v_*)=1$,
the result will follow from the convergence of the sequence
$\{(y_n,v_n)\}$ guaranteed by statement (a).

Adding and subtracting both $f_0(v_*)$ and $f_1(\phi(v))f_0(v_*)$ to
$F_0(v)$ we get,
$$\begin{array}{rl}
|F_0(v)|&\leq
|f_0(\phi(v))-f_0(v_*)|+|f_1(\phi(v))|\,|f_0(v)-f_0(v_*)|+|f_0(v_*)\left(f_1(\phi(v))+1\right)|\\
&\leq W(|\phi(v)-v_*|) +|f_1(\phi(v))| \,W(|v-v_*|) +|f_0(v_*)|\,
V(|\phi(v)-v_*|).
\end{array}
$$
Since $u_*$ is a stable attractor of $u_{n+1}=\phi(u_{n})$ and
$\phi(u)$ is locally contractive at $u=u_*,$ if
$|v-v_*|<\varepsilon$ then $|\phi(v)-v_*|<|v-v_*|$. Using that $W$
and $V$ satisfy $H_3$, we get
$$\begin{array}{rl}
|F_0(v)|&\leq W(|v-v_*|) + |f_1(\phi(v))| \,W(|v-v_*|)
+|f_0(v_*)|\, V(|v-v_*|)\\
&\leq A\, W(|v-v_*|)+ B\, V(|v-v_*|),
\end{array}
$$
where $A:=1+\sup_{|v-v_*|\leq\varepsilon}|f_1(\phi(v))|$ and
$B:=|f_0(v_*)|$. Now we obtain a new function $\widetilde{W}(\nu):=A
W(\nu)+ B V(\nu)$, satisfying $H_1$ for system (\ref{dospassosy}).

On the other hand, using the inequality $|ab-1|=|a(b+1)-(a+1)|\leq
|a|\,|b+1|+|a+1|$, we have
$$\begin{array}{rl}
|F_1(v)-1|&= |f_1(\phi(v))f_1(v)-1|\leq
|f_1(v)|\,|f_1(\phi(v))+1|+|f_1(v)+1|\\
&\leq |f_1(v)|\, V(|\phi(v)-v_*|)+ V(|v-v_*|)\\
&\leq |f_1(v)|\, V(|v-v_*|)+ V(|v-v_*|)\\&\leq C\, V(|\phi(v)-v_*|),
  \end{array}
$$
where $C:=\sup_{|v-v_*|\leq\varepsilon} |f_1(v)|$, obtaining the new
function $\widetilde{V}(\nu):=C V(\nu)$, satisfying $H_2$ for
system~(\ref{dospassosy}).

Now, it is straightforward to prove that $\widetilde{V}(v)$ and
$\widetilde{W}(v)$ satisfy hypothesis $H_3$  and $H_4$, hence each
sequence $\{y_{n}\}$ is convergent, and therefore the sequences
$\{x_{2n}\}$ and $\{x_{2n+1}\}$ are convergent.

Finally, observe that setting
$\ell_e(x_0,u_0):=\lim\limits_{n\to\infty} x_{2n}$ and
$\ell_o(x_0,u_0):=\lim\limits_{n\to\infty} x_{2n+1}$, and using the
continuity of $f_0$ and $f_1$ and equation (\ref{problema1}) we
obtain that $\ell_o=f_0(u_*)-\ell_e$.
\end{proof}

\medskip

\begin{proof}[Proof of Corollary \ref{corolaridiff}]
 If  $u=u_*$ is a hyperbolic
attractor of $\phi$, then for $\varepsilon$ small enough, we have
that $\mu:=\sup_{|u-u_*|<\varepsilon}|\phi'(u)|<1$. Setting
$p_n=\mu^n |u_0-u_*|$, $W(\nu):=A |\nu|$, and
  $V(\nu):=B |\nu|$, where $A$ and $B$ are the supremum in
$|u-u_*|<\varepsilon$ of $|f_0'|$ and $|f_1'|$ respectively, and
using the  mean value theorem,  one gets straightforwardly, that the
Hypothesis $H_1$--$H_4$ are satisfied. So  statements (a) and (b)
follow from Theorem \ref{propoprob1identitatv2}.


Observe that for each point $(x_*,u_*)\in\{u=u_*\}$, the
differential matrix of the map associated to system
(\ref{problema1}), $F(x,u)=(f_0(u)+f_1(u)x,\phi(u))$, is given by
$$
DF(x_*,u_*)=\left(
              \begin{array}{cc}
                1 & f_0'(u_*)+f_1'(u_*)x_* \\
                0 & \phi'(u_*) \\
              \end{array}
            \right)
,
$$
where $|\phi'(u_*)|<1$. Hence from the  stable manifold theorem (see
\cite{BV} or \cite{GH} for instance) there is an invariant
$\mathcal{C}^1$ curve, transversal to $\{u=u_*\}$ at $(x_*,u_*)$,
such that any initial condition on this curve gives rise to a
solution with limit $(x_*,u_*)$.
\end{proof}

From the above proof, we notice that the statements (a) and (b) in
Corollary \ref{corolaridiff} are also satisfied in the case that
$\phi$ is not differentiable but $|\phi(u)-u_*|\leq \mu |u-u_*|$
with $0<\mu<1$, for those values of $u$ such that
$|u-u_*|<\varepsilon$, for $\varepsilon>0$ small enough.

\medskip

\section{Applications}\label{Aplications1}
\subsection{Linear quasi-homogeneous maps}\label{quasi1}

 We say that $f:\R^2\longrightarrow \R$ is a {\it quasi-homogeneous
 function} with weights $(\alpha,\beta)$ and quasi-degree $d$ if
 $$f(\lambda^\alpha\,x,\lambda^\beta\,y)=\lambda^d\,f(x,y),\mbox{ for all }\lambda>0.$$ Notice that if $f$ is a quasi-homogeneous
 function with weights $(\alpha,\beta)$ then, for all constant $c\ne
 0$, $f$ is also a quasi-homogeneous
 function  with weights $(c\alpha,c\beta).$
Hence we will consider $(\alpha,\beta)\in\Z^2$ with
gcd$(\alpha,\beta)=1.$

 We say that $F=(f,g):\R^2\longrightarrow \R^2$ is a {\it quasi-homogeneous
map} with weights $(\alpha,\beta)$ and quasi-degree $d$ if  $f(x,y)$
and $g(x,y)$ are  quasi-homogeneous functions with weights
$(\alpha,\beta)$ and quasi-degrees $d\alpha$ and $d\beta$,
respectively. This section deals with quasi-homogeneous maps of
quasi-degree $1$ (which are called \emph{linear quasi-homogeneous
maps}) with weights of different sign.

Assuming that $\alpha>0$ and  $\beta<0,$ from the definition it
follows that in the class of analytic maps, any linear
quasi-homogeneous map takes the form
$$F(x,y)=\left(x\,p(x^{-\beta}y^{\alpha}),y\,q(x^{-\beta}y^{\alpha})\right),$$
where $p(z)$ and $q(z)$ are analytic functions. These maps preserve
the fibration given by $\mathcal{F}=\{x^{-\beta}y^{\alpha}=h,$
$h\in\R\}$, since they send the curves $x^{-\beta}y^{\alpha}=h$ to
the curves $x^{-\beta}y^{\alpha}=h\,p(h)^{-\beta}q(h)^{\alpha}.$ The
dynamical system associated to $F$ is:
$$
\left\{
 \begin{array}{rl}
x_{n+1}&=x_n\,p(x_n^{-\beta}y_n^{\alpha}),\\
y_{n+1}&=y_n\,q(x_n^{-\beta}y_n^{\alpha}).
  \end{array}
\right.
$$

Applying the transformation $(x,y)\rightarrow
(x,x^{-\beta}y^{\alpha})$ and calling $u=x^{-\beta}y^{\alpha},$the
map $F(x,y)$ is transformed into
$$\tilde
F(x,u)=\left(x\,p(u),u\,p(u)^{-\beta}\,q(u)^{\alpha}\right),$$ which
is a triangular map, whose corresponding system is:
$$
\left\{
  \begin{array}{rl}
x_{n+1}&=x_n\,p(u_n),\\
u_{n+1}&=u_n\,p(u_n)^{-\beta}q(u_n)^{\alpha}.
  \end{array}
\right.
$$
Notice that such a system is of the form (\ref {problema1}) and the
fiber $u=0$ is invariant. In order to apply Proposition
\ref{propoprob1attractor} to study the basin of attraction of the
origin, we need that $|p_0|<1$ and that $u=0$ is an attractive point
of the subsystem $u_{n+1}=u_n\,p(u_n)^{-\beta}q(u_n)^{\alpha},$
which is guaranteed by the condition
$|p_0^{-\beta}\,q_0^{\alpha}|<1,$ where $p_0=p(0)$ and $q_0=q(0).$
So, in this case we have $\lim\limits_{n\to\infty}(x_n,u_n)=(0,0)$
for all $u_0\in\mathcal{B},$ where $\mathcal{B}$ is the basin of
attraction of $u=0,$ and any arbitrary value of $x_0$.

On the other hand, if we perform the similar change
$(x,y)\rightarrow (x^{-\beta}y^{\alpha},y)$ and we call
$u=x^{-\beta}y^{\alpha},$ the map $F(x,y)$ is transformed into
another triangular map
$$\bar
F(u,y)=\left(u\,p(u)^{-\beta}\,q(u)^{\alpha},y\,q(u)\right),$$ whose
associated system is:
$$
\left\{
  \begin{array}{rl}
u_{n+1}&=u_n\,p(u_n)^{-\beta}q(u_n)^{\alpha},\\
y_{n+1}&=y_n\,q(u_n).
  \end{array}
\right.
$$

As before we can apply Proposition \ref{propoprob1attractor}. Thus,
if $|q_0|<1$ and $|p_0^{-\beta}\,q_0^{\alpha}|<1,$ then
$\lim\limits_{n\to \infty}(u_n,y_{n})= (0,0)$ for all
$u_0\in\mathcal{B},$ where $\mathcal{B}$ is the basin of attraction
of $u=0,$ for the subsystem
$u_{n+1}=u_n\,p(u_n)^{-\beta}\,q(u_n)^\alpha$ and $y_0$ is
arbitrary. Hence, we get:

\begin{propo}\label{hom}
Assume that $|p_0|<1$ and $|q_0|<1$ and let $\mathcal{B}$ be the
basin of attraction of $u=0$ for the one-dimensional system
$u_{n+1}=u_n\,p(u_n)^{-\beta}q(u_n)^{\alpha}.$ Then for all
$(x_0,y_0)$ such that $x_0^{-\beta}\,y_0^{\alpha}\in\mathcal{B},$
the sequence $\{(x_n,y_n)\}$ tends to $(0,0)$ as $n\to\infty.$
\end{propo}

\medskip

\noindent \textsl{Example F.} As a particular easy case we can take
$(\alpha,\beta)=(1,-1)$ and
\begin{equation}\label{qh3}
\left\{
 \begin{array}{rl}
x_{n+1}&=x_n\,(a+bx_ny_n),\\
y_{n+1}&=y_n\,(c+dx_ny_n),
  \end{array}
\right.
\end{equation}
which is transformed in:
$$
\left\{
  \begin{array}{rl}
x_{n+1}&=x_n\,(a+bu_n),\\
u_{n+1}&=u_n\,(a+bu_n)\,(c+du_n),
  \end{array}
\right.
$$
with $|a|<1$ and $|c|<1.$ To determine the basin of attraction of
$u=0$ for arbitrary values of $a,b,c,d$ is not an easy task. In
fact, the dynamics of the one dimensional system $u_{n+1}=\phi(u_n)$
with $\phi(u):=u(a+bu)(c+du)$ can be very complicated. For instance,
if such map has a $3$-periodic point, then it has periodic points
for all the periods. It means that the \emph{intrafibration}
dynamics of the hyperbolas $xy=h$ for system (\ref{qh3}) can be
complicated.

One case for which the dynamics of $u_{n+1}=\phi(u_n)$ is simple is
just when we have three fixed points with alternating stability. For
instance, if we consider   $a=2/(3d),b=-1/(6d)$ and $c=d,$ then we
get three fixed points $u=0,u=1$ and $u=2.$ Moreover
$\phi'(0)=\phi'(2)=2/3$ so $0$ and $2$ are attractive points  and
$\phi'(1)=7/6$ and $1$ is a repelling point.

There are three preimages of the fixed point $1,$ that are $1,$
$p_1=1-\sqrt{7}$ and $p_2=1+\sqrt{7}$.
 It easily follows that $I_1:=(p_1,1)$ is
contained in $\mathcal{B}_0,$
 the basin of attraction of $u=0,$ while $J_1:=(1,p_2)$ is contained in $\mathcal{B}_2,$
 the basin of attraction of $u=2.$ Setting
$I_{2}=\phi^{-1}(I_1)\setminus I_1$ and  $I_{n+1}=\phi^{-1}(I_n)$
for $n>2$; and $J_{2}=\phi^{-1}(J_1)\setminus J_1$,
$J_{n+1}=\phi^{-1}(J_n)$ we have that these intervals are interlaced
as
$$\ldots
I_7,\,J_6,\,I_5,\,J_4,\,I_3,\,J_2,\,I_1,\,J_1,\,I_2,\,J_3,\,I_4,\,J_5,\,I_6,\,\ldots$$
where the right extreme of each interval $J_{2k}$ equals the left
one of $I_{2k-1}$ and the right extreme of each interval $J_{2k-1}$
equals the left one of $I_{2k}.$ These boundary points are exactly
the preimages of the fixed point $u=1.$ Furthermore, these intervals
have decreasing length, tending to the two-periodic points
$q_\pm=1\pm\sqrt{13}.$

It is very easy to prove that $\phi$ sends  $\R\setminus[q_-,q_+]$
to itself. Collecting all the above observations we get:

\begin{lem}\label{l10}
Consider the one dimensional discrete system generated by
$$\phi(u)=\frac{1}{6}\,u\,\left(4-u\right)\,\left(1+u\right).$$ Let
$\mathcal{B}_0$(resp. $\mathcal{B}_2$ or $\mathcal{B}_\infty$) be
the basin of attraction of $u=0$ (resp. $u=2$ or infinity) and let
$\mathcal{O}_{1}:= \cup_{n\in\N}\phi^{-n}(1)$ be, the set of
preimages of the repulsive fixed point $u=1.$ Then
$$\mathcal{B}_0=\cup_{n\in\N}\phi^{-n}(I_1),\quad\mathcal{B}_2=\cup_{n\in\N}\phi^{-n}(J_1)\quad\text{and}\quad
\mathcal{B}_0\cup\mathcal{B}_2\cup \mathcal{O}_{1}=(q_-,q_+),$$
where $q_\pm=1\pm\sqrt{13}$ is the unique orbit of $\phi(u)$ with
minimal period $2$. Moreover
$$\mathbb{R}=\mathcal{B}_0\cup\mathcal{B}_2\cup \mathcal{O}_{1}\cup
\mathcal{B}_\infty\cup\{p_-,p_+\}.$$
\end{lem}

Once the basin of attraction of $u=0$ is determined, we can come
back to system (\ref{qh3}).

\begin{propo} Let $\mathcal{B}_0$  be the basin of attraction of $u=0$  for
$\phi(u)=u\left(4-u\right)\left(1+u\right)/6.$  For $d\ne0,$
consider the system
\begin{equation}\label{propoexemple}
\left\{
 \begin{array}{rl}
x_{n+1}&=\displaystyle{\frac{1}{6d}}\,x_n\,(4-x_ny_n),\\
y_{n+1}&=d\,y_n\,(1+x_ny_n).
  \end{array}
\right.
\end{equation}
\begin{enumerate}[(i)]
 \item For  $2/3<|d|<1$,
\begin{enumerate}
  \item[(a)] If   $x_0y_0\in\mathcal{B}_0$ then
  $\lim\limits_{n\to\infty}(x_n,y_n)=(0,0);$ see Figure~1.
  \item[(b)] If $x_0y_0\not \in\mathcal{B}_0$ then $\lim\limits_{n\to\infty}|x_n|+|y_n|=\infty.$
\end{enumerate}

\item For $|d|\in\mathbb{R}^+\setminus [2/3,1],$ if $x_0y_0\ne0$ then $\lim\limits_{n\to\infty}|x_n|+|y_n|=\infty.$
\end{enumerate}
\end{propo}

\begin{proof}
(i) By using Lemma~\ref{l10},  assertion $(a)$ is a consequence of
Proposition \ref{hom}.

To prove (b), assume that $x_0y_0\not\in\mathcal{B}_0.$ Then, again
by Lemma~\eqref{l10}, the sequence $\{u_n\}$, where $u_n=x_ny_n$,
has four possibilities
 when $n$ goes to infinity: either it tends to $2$, or  after some iterates it is constant equal to 1, or it takes the
 values
 $q_-,q_+$, or it tends to infinity. Using $y_{n+1}=dy_n\left(1+u_n\right)$, we see that in the first case, when
$n$ is big enough $y_{n+1}\simeq 3d\,y_n$. Since $2/3<|d|<1$, it
follows that $|y_n|$ tends to $\infty$ when $n$ goes to infinity.
The other  cases follow similarly.

The proof of (ii) is a consequence of Lemma~\ref{l10} and the
dynamics of~\eqref{propoexemple} on the invariant sets $xy=0, xy=1,$
$xy=2$ and $(xy-q_-)(xy-q_+)=0.$ For instance, on the third one,
$$
\left\{
 \begin{array}{rl}
x_{n+1}&=\displaystyle{\frac{1}{3d}}\,x_n,\\
y_{n+1}&=3d\,y_n,
  \end{array}
\right.
$$
and clearly the orbits go towards infinity on it.
\end{proof}

Notice that system \eqref{propoexemple} presents  interesting
bifurcation when $|d|\in\{2/3,1\}.$ In particular, as a consequence
of the shape of $\mathcal{B}_0$ explained above, when $2/3<|d|<1$,
the basin of attraction of $(0,0)$ of this system
 is formed by the union of  infinitely many
disjoint hyperbolic-shaped bands which accumulate to the hyperbolas
$xy=q_\pm$; see Figure 1.

\medskip

\begin{center}
\includegraphics[scale=0.30]{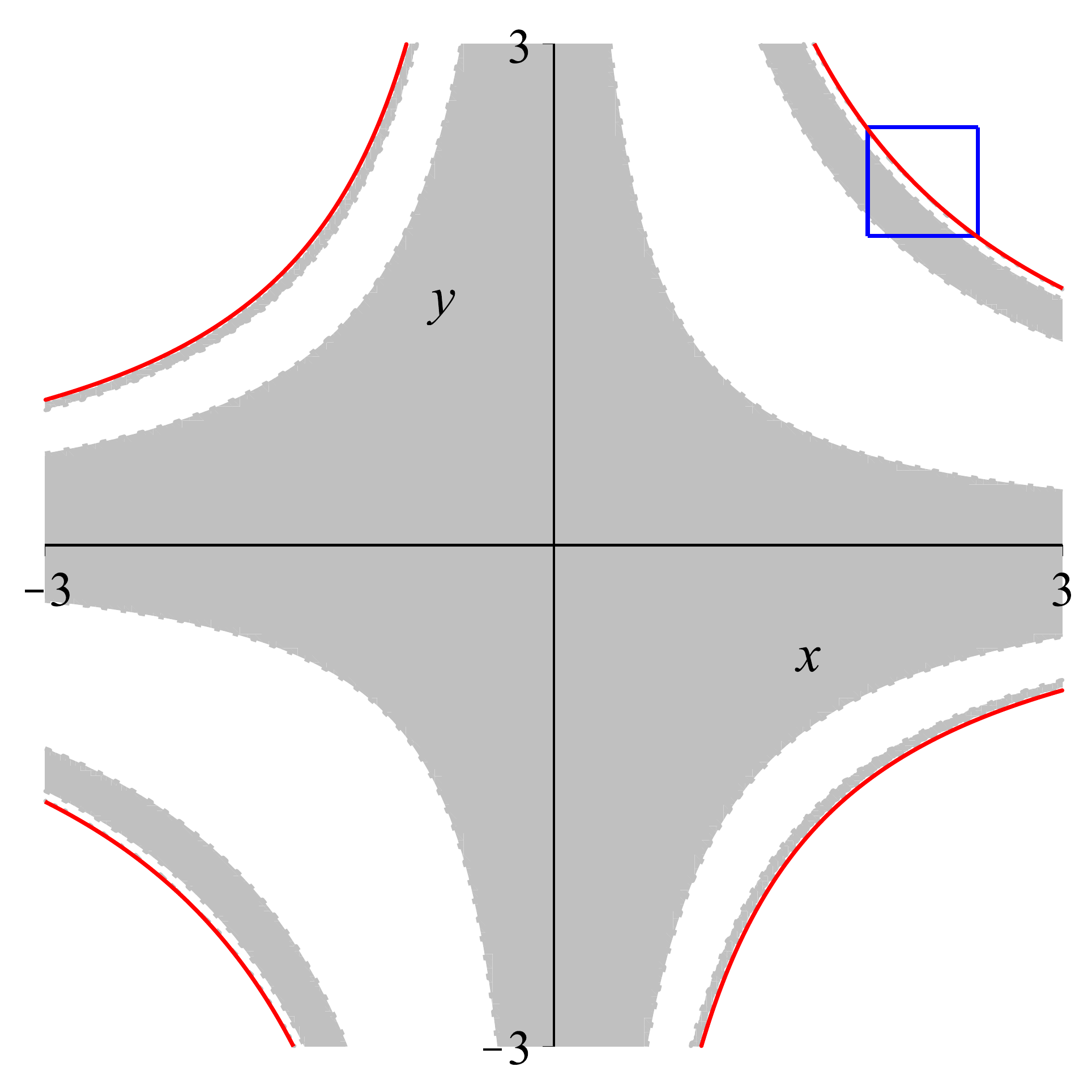}\hspace{1cm}
\includegraphics[scale=0.30]{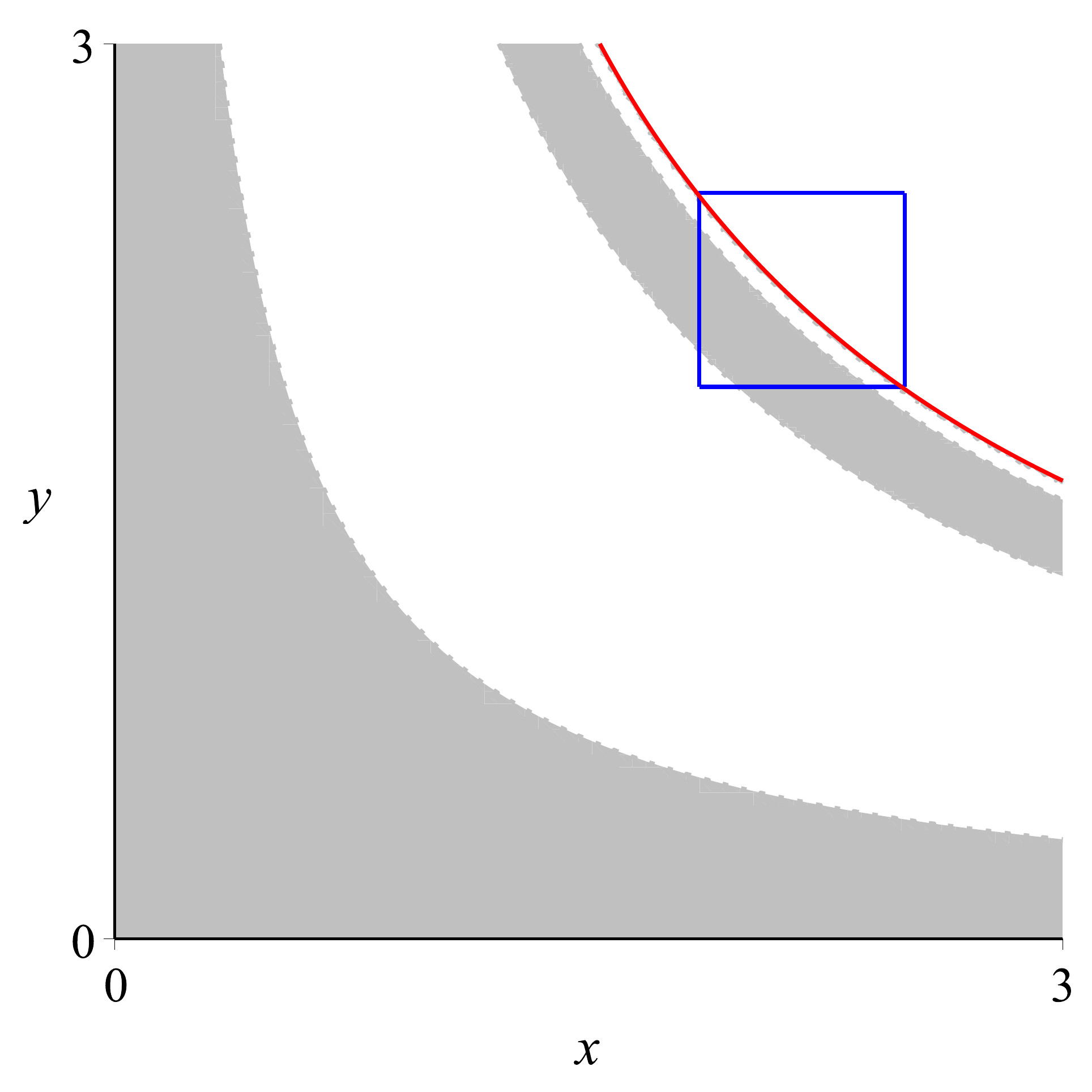}\\

\includegraphics[scale=0.30]{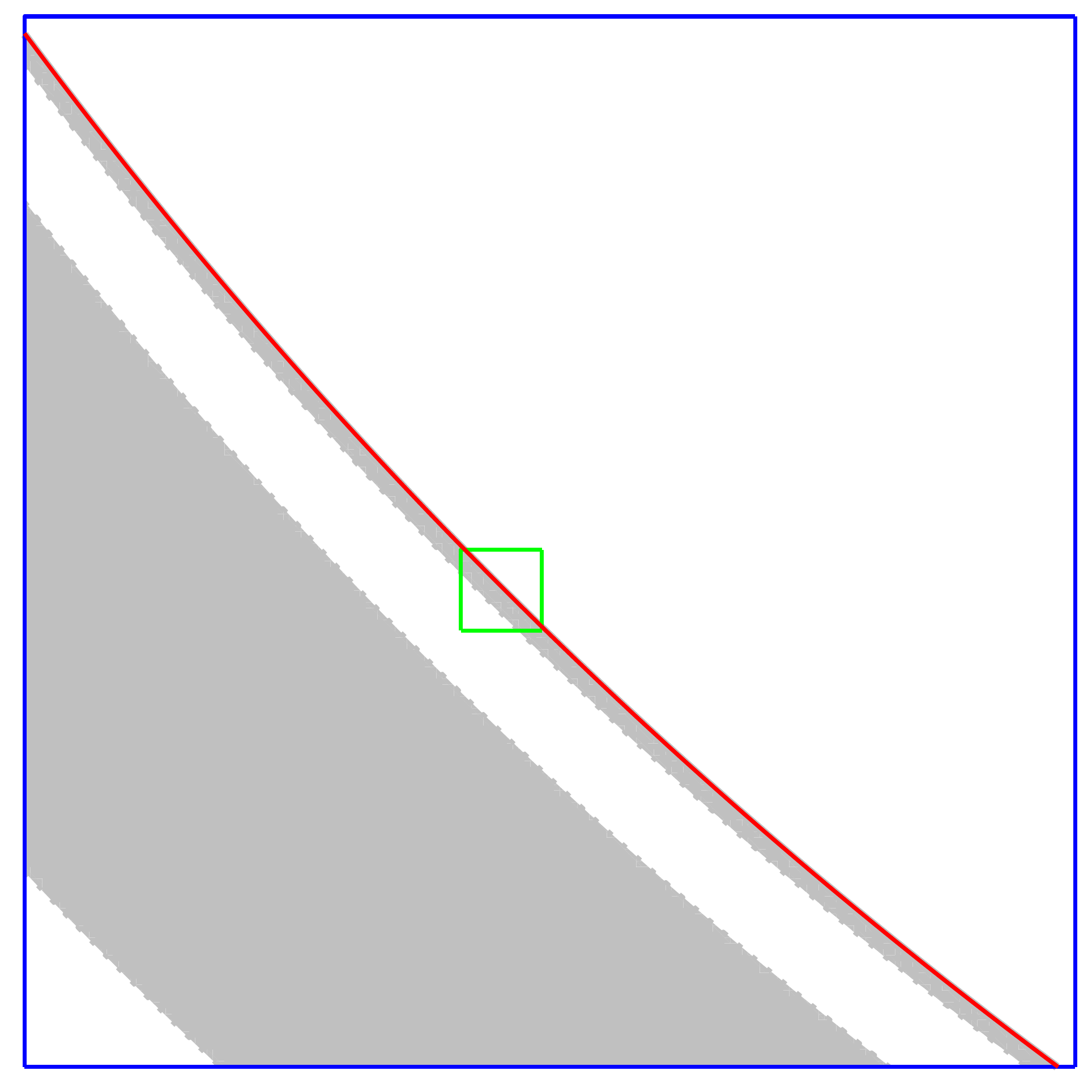}\hspace{1cm}
\includegraphics[scale=0.30]{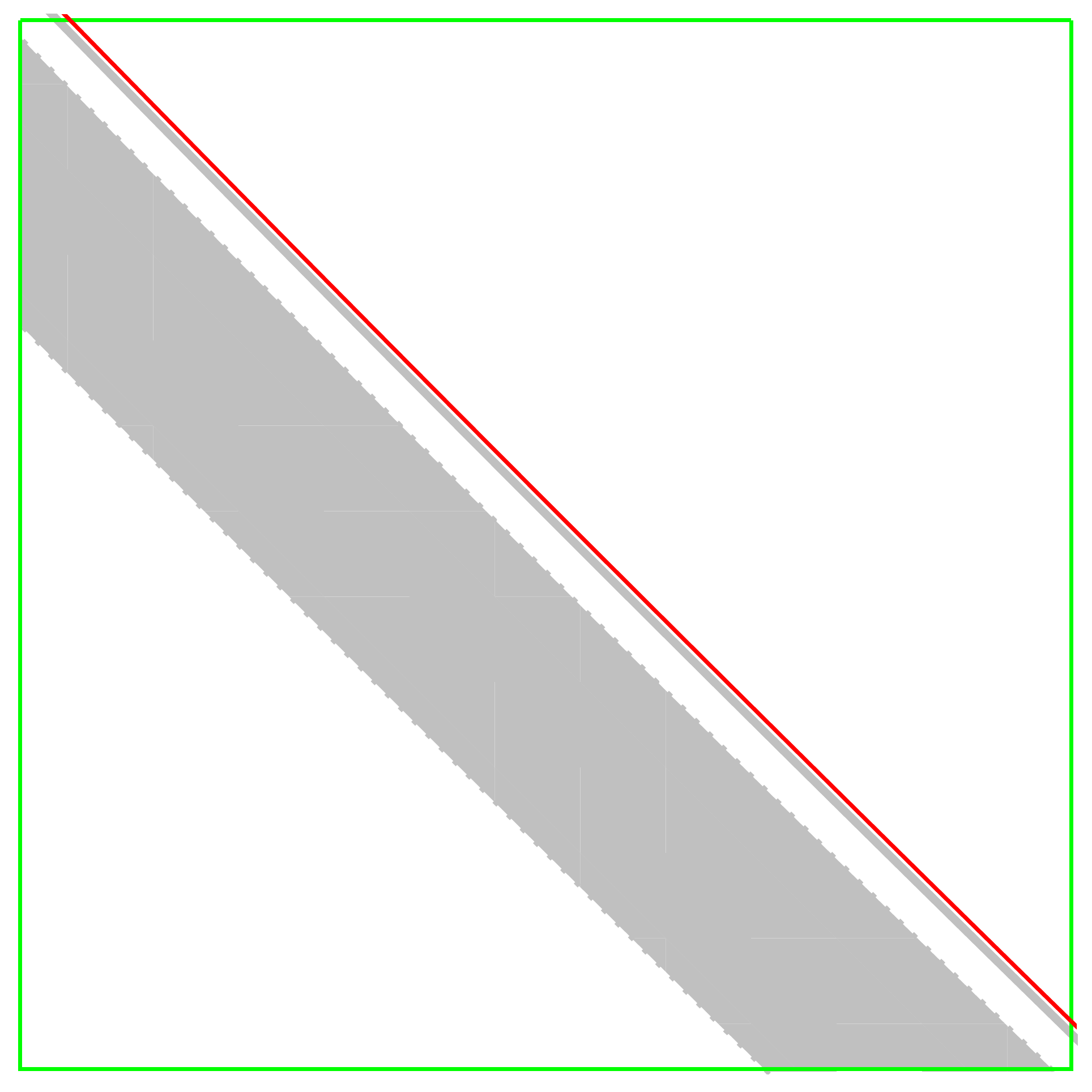}
\end{center}
\begin{center}
Figure 1: Details of the basin of attraction of $(0,0)$ of system
(\ref{propoexemple}) for $2/3<|d|<1,$\\  shaded in grey.
\end{center}

\medskip

\subsection{Second order multiplicative  difference
equations}\label{secondordermultiplicative}

We consider the next family of second order multiplicative
difference equations
\begin{equation}\label{g1}
x_{n+2}=x_ng(x_nx_{n+1}),
\end{equation}
where $g:\mathcal{U}\to\R$ is a $\mathcal{C}^1$ function defined in
an open set of $\mathcal{U}\subseteq\R$.  Multiplying both sides
of~(\ref{g1}) by $x_{n+1}$, and setting $u_n=x_n x_{n+1}$ we get
that Equation (\ref{g1}) can be written as
\begin{equation}\label{g3}
\left\{
  \begin{array}{rl}
x_{n+1}&=\frac{u_n}{x_n},\\
u_{n+1}&=u_n\,g(u_n),
  \end{array}
\right.
\end{equation}
which has the associated map
$$F(x,u)=\left(\frac{u}{x},u\,g(u)\right).$$
If we consider the map $F^2(x,u)$ we obtain
$$F^2(x,u)=\left(g(u)\,x,u\,g(u)\,g(u\,g(u))\right).$$
Hence, by calling $z_n=x_{2n}$ and $v_n=u_{2n}$ (respectively
$z_n=x_{2n+1}$ and $v_n=u_{2n+1}$) we get the system
\begin{equation}\label{g4}
\left\{
  \begin{array}{rl}
z_{n+1}&=g(v_n)\,z_n,\\
v_{n+1}&=v_n\,g(v_n)\,g(v_n\,g(v_n)),
  \end{array}
\right.
\end{equation}
which is of the form (\ref{problema1}). Hence we can apply
Proposition \ref{propoprob1attractor} and Corollary
\ref{corolaridiff} to system (\ref{g4}) and study the behavior of
$(x_{2n},u_{2n})$ (respectively $(x_{2n+1},u_{2n+1})$).  Notice that
when we consider the initial conditions $(x_0,u_0)$ (resp.
$(x_1,u_1)$ and we apply system (\ref{g4}) iteratively we get
$(x_{2n},u_{2n})$ (resp. $(x_{2n+1},u_{2n+1})$).

\medskip

\noindent \textsl{Example G.} Among the recurrences of type
(\ref{g1}), we are going to consider the one given by
\begin{equation}\label{Bliz}
    x_{n+2}=\frac{x_n}{a+bx_nx_{n+1}},\,a,b\in\R.
\end{equation}
The global behavior of Equation (\ref{Bliz}) is now completely
understood after the work done by Bajo and Liz \cite{BL}. Their
approach is based on the computation of the explicit solutions of
(\ref{Bliz}). Previously, particular cases of this equation have
been studied by a large numbers of experts, see
\cite{A,ACL1,ACL2,ASM,C} and \cite{Stev0}, the reader is also
referred to \cite{BL} for a summary of these previous references
concerning this equation.

Equation (\ref{Bliz}) has the form (\ref{g1}), where
$g(u)=1/(a+bu)$. Furthermore, it corresponds with the systems of
type (\ref{g3}):
$$
\left\{
  \begin{array}{rl}
x_{n+1}x_n&=u_n,\\
u_{n+1}&=\frac{u_n}{a+b u_{n}}.
  \end{array}
\right.
$$
This system has two invariant fibers given by $\{u=0\}$ and
$\{u=(1-a)/b\}$ whose stability is determined by the parameter $a$,
since setting $\phi(u)=ug(u)=u/(a+bu)$ we have $\phi'(0)=1/a$ and
$\phi'((1-a)/b)=a$. The global dynamics of equation (\ref{Bliz}) in
its \emph{good set} (that is, the set of all initial conditions for
which  the dynamical system is well defined, \cite{Ladas})
$\mathcal{G}=\R\setminus\{\cup_{n\geq
0}\phi^{-n}(-a/b)\}$\footnote{In another framework we could consider
this recurrence defined in $\bar{\R}=[-\infty,\infty]$, and then the
good set would be $\bar{\R}$.}, is determined by the graph of
$\phi$. First it is necessary to know the dynamics of the recurrence
$u_{n+1}=\phi(u_n)$ which, since $\phi$ is a M\"obius
transformation, is well known. For instance, using \cite[Corollary
7]{CGM06} we get:
\begin{lem}\label{moeb}
Consider the real one-dimensional recurrence given by
\begin{equation}\label{moeb2}
    u_{n+1}=\frac{u_n}{a+b u_{n}}.
\end{equation}
Then the points $u=0$ and $u=(1-a)/b$ are fixed. Furthermore setting
$\mathcal{G}=\R\setminus\{\cup_{n\geq 0}\phi^{-n}(-a/b)\}$ we have:
\begin{enumerate}
  \item[(a)] If $|a|>1$, then for any initial condition $(1-a)/b\neq u_0\in\mathcal{G}$, $\lim\limits_{n\to\infty}u_n=0$.
  \item[(b)] If $|a|<1$, then for any initial condition $0\neq u_0\in\mathcal{G}$, $\lim\limits_{n\to\infty}u_n=(1-a)/b$.
  \item[(c)] If $a=-1$, then for any initial condition in $\mathcal{G}$, the sequence $\{u_n\}$ is $2$-periodic.
  \item[(d)] If $a=1$, then for any initial condition in $\mathcal{G}$,  $\lim\limits_{n\to\infty}u_n=0$.
\end{enumerate}
\end{lem}

Hence as consequence of the above lemma, Proposition
\ref{propoprob1attractor} and Corollary \ref{corolaridiff}, as well
as some other ad hoc arguments, we give an alternative proof of the
result in \cite{BL}:

\begin{teo}\label{propoliz}
Consider Equation (\ref{Bliz}). Then the following statemens hold:
\begin{enumerate}
\item[(a)]  If $x_0=x_1=0$, then the sequence $x_n=0$ for all $n\in\N$.
  \item[(b)] If $|a|>1$, then for any initial condition $x_0,x_1$ such that $(1-a)/b\neq x_0x_1\in\mathcal{G}$,
  $\lim\limits_{n\to\infty}x_n=0$. If $x_0x_1=(1-a)/b$ then $\{x_n\}$ is
  $2$-periodic.
  \item[(c)] If $|a|<1$, then for any initial condition $x_0,x_1$ such that
  $x_0x_1\in\mathcal{G}$ we have: if $x_0x_1\neq 0$ and $x_0x_1\neq (1-a)/b$
   then $\{x_n\}$ tends to a   $2$-periodic orbit
   $\{\ell_0(x_0,x_1),\ell_1(x_1,x_2)\}$ such that
   $\ell_0(x_0,x_1)\ell_1(x_1,x_2)=u_*$; if $x_0x_1=(1-a)/b$ then $\{x_n\}$ is
  $2$-periodic; and if $x_0x_1=0$ then $\lim\limits_{n\to\infty}|x_n|=\infty$.
 \item[(d)] If $a=-1$, then for any initial condition $x_0,x_1$ such that
  $x_0x_1\in\mathcal{G}$ we have: if $x_0x_1\neq 0$ and $x_0x_1\neq 2/b$
   then the solution $\{x_n\}$ is unbounded; if $x_0x_1\neq 2/b$ then $\{x_n\}$ is
  $2$-periodic; and if $x_0x_1=0$ then $\{x_n\}$ is
  $4$-periodic.
  \item[(e)]  If $a=1$, then for any initial condition $x_0,x_1$ such that $0\neq x_0x_1\in\mathcal{G}$,
  $\lim\limits_{n\to\infty}x_n=0$. If $x_0x_1=0$ then $\{x_n\}$ is
  $2$-periodic.
\end{enumerate}
\end{teo}

\begin{proof} Statement (a) is trivial.

 In order to prove (b) we consider $|a|>1.$ Then, by Lemma \ref{moeb} (a) $u=0$
is an attractor of the recurrence (\ref{moeb2}) in
$\mathcal{G}\setminus\{u=(1-a)/b\}$. We now note that for equation
(\ref{Bliz}), system (\ref{g4}) writes as
\begin{equation}\label{dup}
\left\{
  \begin{array}{rl}
z_{n+1}&=\frac{1}{a+bv_n}z_n,\\
v_{n+1}&=\frac{v_n}{a^2+b(1+a)v_{n}}.
  \end{array}
\right.
\end{equation}
Applying Proposition \ref{propoprob1attractor} to system (\ref{dup})
we deduce that $z_n\to 0.$ It implies that $x_{2n}\to 0$ and
$x_{2n+1}\to 0$ too. Hence, for each initial condition such that
$(1-a)/b\neq x_0x_1\in\mathcal{G}$ $\lim\limits_{n\to\infty}x_n=0$.
On the other hand, substituting $x_0x_1=(1-a)/b$ in Equation
(\ref{Bliz}) we get $x_2=x_0$, obtaining a $2$-periodic orbit.

(c) If $|a|<1$, then by Lemma \ref{moeb} (b), for all $u_0\ne
u_*:=(1-a)/b$ the sequence $u_n\to u_*.$  Since $f_1(u_*)=1$ and
$u_*$ is a hyperbolic attractor of (\ref{moeb2}),  we can use
Corollary \ref{corolaridiff} to assert that the sequence $v_n$
converges to a point which depend on the initial condition. Then, if
we take the initial condition $(z_0,v_0)=(x_0,u_0)$  (resp.
$(z_0,v_0)=(x_1,u_1)$) we have that
$\lim_{n\to\infty}u_{2n}=\lim_{n\to\infty}v_n=\overline{l_0}(x_0,u_0):=l_0(x_0,x_1)$
(resp.
$\lim_{n\to\infty}u_{2n+1}=\lim_{n\to\infty}v_n=\overline{l_1}(x_1,u_1):=l_1(x_1,x_2)$).
Then, since $u_n=x_nx_{n+1},$ the condition
$l_0(x_0,x_1)\,l_1(x_1,x_2)=u_*$ must be satisfied. The other
assertions of statement $(c)$ are easily deduced from the equation
(\ref{Bliz}).

(d) If $a=-1$ and $x_0x_1\ne 0\,,\,x_0x_1\ne 2/b$ then
$$u_{n+2}=u_n,\quad
x_{2n}=\frac{x_0}{(bx_0x_1-1)^n}\quad\text{and} \quad
x_{2n+1}=x_1(bx_0x_1-1)^n.$$ Hence $x_n$ is unbounded. If
$x_0x_1=2/b$, then $x_{2n}=x_0\,,\,x_{2n+1}=x_1$ and $\{x_n\}$ is
2-periodic. If $x_0=0$, then  $x_{2n}=0\,,\,x_{2n+1}=x_1\,(-1)^n$
and $\{x_n\}$ is 4-periodic. If $x_1=0$, then
$x_{2n}=(-1)^n\,x_0\,,\,x_{2n+1}=0$ and $\{x_n\}$ is 4-periodic.

(e) Similarly, if $a=1$ and $x_0x_1= 0$ then $\{x_n\}$ is
$2$-periodic. Consider now $x_0x_1\neq 0.$ Equation (\ref{g4}) is

\begin{equation}\label{nousistema30bis}
\left\{
  \begin{array}{rl}
z_{n+1}&=\left(\frac{1}{1+bv_n}\right)z_{n},\\
v_{n+1}&=\frac{v_{n}}{1+2bv_{n}}.
  \end{array}
\right.
\end{equation}
 An straightforward computation shows
that the second component of system (\ref{nousistema30bis}) has the
following explicit solution:
\begin{equation}\label{vn}
v_{n}=\frac{v_{0}}{1+2\,b\,v_{0}\,n}.
\end{equation}
Hence the  first equation of  system (\ref{nousistema30bis}) is
$$
z_{n+1}=\left( \frac{1}{1+\frac{bv_{0}}{2\,b\,v_{0}\,n+1}}
\right)z_{n},
$$
and we can construct the explicit solution
$$
z_{n+1}=\left(\prod\limits_{j=0}^n
\frac{1}{1+\frac{bv_{0}}{2\,b\,v_{0}\,j+1}}\right)\,z_{0}.
$$

So if $n>n_0$ for a suitable $n_0$, we can take logarithms in the
above equations, obtaining:
$$
\ln|z_{n+1}|=- \sum\limits_{j=0}^n
\ln\left|1+\frac{b\,v_{0}}{2\,b\,v_{0}\,j+1}\right|\,+\ln|z_{0}|\sim
- \sum\limits_{j=0}^n
\left|\frac{bv_{0}}{2\,b\,v_{0}\,j+1}\right|.$$ Observe that, for
any value of $b$ and $v_{0}$ we have $\lim\limits_{n\to\infty}
-\sum\limits_{j=0}^n
\left|\frac{bv_{0}}{2\,b\,v_{0}\,j+1}\right|=-\infty,$ hence
$\lim\limits_{n\to \infty} z_{n}=0$, and the result follows.
\end{proof}

\medskip

The next examples show that the approach in this section can be
applied to other families of difference equations not in the class
(\ref{g1}). We will not develop the analysis in this paper, our main
purpose is only to illustrate a range of applications of the results
in Section \ref{main}.

\medskip

\noindent \textsl{Example H: A third order equation.} The exact
solutions of the third order equation
$$
    x_{n+3}=\frac{x_{n+2}x_n}{x_{n+1}(a+bx_{n+2}x_n)},
$$
have been obtained in \cite{Ibr}. A complete analysis of the
dynamics associated to this equation can be done using a similar
approach as above. Setting $\phi(u)=u/(a+bu)$, we have that the
subsequences $\{x_{2k}\}$ and $\{x_{2k+1}\}$ can be studied using
the system
$$
\left\{
  \begin{array}{rl}
y_{n+1}y_{n}&=v_{n},\\
v_{n+1}&=\phi^2(v_{n})=\frac{v_{n}}{b(a+1)v_{n}+a^2},
  \end{array}
\right.
$$ where $y_{n}=x_{2n+i}$, and $v_{n}=u_{2n+i}$ for $i=0,1$.
Hence, as in the case of Example G, the sequences $\{x_{2k+i}\}$ can
be straightforwardly characterized using the behavior of the real
M\"oebius recurrence $v_{n+1}=v_{n}/(b(a+1)v_{n}+a^2)$.

\medskip

\noindent \textsl{Example I.} Consider the equation
$$
    x_{n+2}=\frac{x_n^\gamma g(x_{n+1}x_n^\gamma)}{x_{n+1}^{\gamma-1}}.
$$
with $x_n\in\R^+$, $\gamma\in\R$ and being $g$ a continuous positive
function. By multiplying both sides of the equation by $x_{n+1}$;
setting $u_n=x_{n+1}x_n^\gamma$, and taking $y_n=\ln x_n$ we get
 the system of type
(\ref{problema1})
$$
\left\{
  \begin{array}{rl}
y_{n+1}&=\ln(u_n)-\gamma y_n,\\
u_{n+1}&=u_ng(u_n),
  \end{array}
\right.
$$
that can be studied using the results in Section \ref{main}, if
$|\gamma|\leq 1$ and there is a limit fiber $\{u=u_*\neq 0\}$ with
the fast enough convergence property.

\subsection{Higher order multiplicative   difference equations}\label{higher}

The above approach to study Equation (\ref{g1}) can be applied to
higher order multiplicative-type of difference equations. Consider
the order $k$ difference equation:
\begin{equation}\label{gk}
    x_{n+k}=x_n\,g(x_n x_{n+1}\ldots x_{n+k-1}).
\end{equation}
Some straightforward computations using the associated map
$$
    F(x_0,\ldots,x_{k-1})=(x_1,\ldots,x_{k-1},x_0g(x_0\cdots
    x_{k-1})),
$$
show that the sets $x_0x_1\cdots x_{k-1}=0$ and $x_0x_1\cdots
x_{k-1}=u_*$ for  $u^*\neq 0$ such that $g(u^*)=1$ (if it exists)
are invariant, and also lead to the following result:
\begin{lem}\label{lemafibresinvariantsk}
Consider the equation (\ref{gk}), being $g$ a continuous function
defined in an open set $\mathcal{U}\subseteq\R$ such that
$0\in\mathcal{U}$. Given some initial conditions $x_0,\ldots,
x_{k-1}$ we have:
\begin{enumerate}
  \item[(i)]  If $x_i=0$ for all $i=0,\ldots, k-1$ then $x_n=0$ for all $n\in\N$.
  \item[(ii)]  Suppose that $x_0x_1\cdots x_{k-1}=0$ and $x_0^2+\cdots+x_{k-1}^2\neq
  0$. If additionally $g(0)=-1$ then $\{x_n\}$ is $2k$-periodic; if $g(0)=1$ then $\{x_n\}$ is
  $k$-periodic; if $|g(0)|<1$ then $\lim\limits_{n\to\infty} x_n=0$;
  and   if  $|g(0)|>1$ then $\lim\limits_{n\to\infty} |x_n|=\infty$.
  \item[(iii)] If there exists $u^*\neq 0$ in $\mathcal{U}$ such that $g(u^*)=1$,  and $x_0x_1\cdots x_{k-1}=u^*$ then $\{x_n\}$ is
  $k$-periodic (non-minimal).
\end{enumerate}
\end{lem}

\begin{propo}\label{corolgk0} Consider the equation  (\ref{gk}) being
$g$ a continuous function  defined in an open set
$\mathcal{U}\subseteq\R$. Assume  that $0\in\mathcal{U}$ and
$|g(0)|<1$, then $u=0$ is an attractor for the system
$u_{n+1}=u_ng(u_n)$, and for all initial condition
$x_0,\ldots,x_{k-1}$ such that $0\neq u_0=x_0\ldots x_{k-1}$ is in
the basin of attraction of $\{u=0\}$ we have
$\lim\limits_{n\to\infty}x_n=0$.
\end{propo}

\begin{proof} Multiplying both sides of (\ref{gk}) by $x_{n+k-1}\cdots x_{n+1}$
and setting $u_n=x_n x_{n+1}\cdots x_{n+k-1}$, we obtain that
equation (\ref{gk}) can be written as
$$
\left\{
  \begin{array}{rl}
x_{n+k-1}\cdots x_{n+1}x_n&=u_n,\\
u_{n+1}&=u_ng(u_n).
  \end{array}
\right.
$$
Hence, we have $x_{n+k}\cdots x_{n+1}=u_{n+1}$ and $x_{n+k-1}\cdots
x_n=u_n$. Dividing these expressions we have
$$
x_{n+k}=g(u_n)x_n,
$$
 From the above equation, and after renaming $z_{i,n+1}=x_{n+k}$ and $z_{i,n}=x_n$ we get  $k$
systems of type (\ref{problema1}) associated to each initial
condition $(x_i,u_i)$ for $i=0,\ldots,k-1$.
\begin{equation}\label{nousistemak30}
\left\{
  \begin{array}{rl}
z_{i,n+1}&=g(v_{i,n})z_{i,n},\\
v_{i,n+1}&=\varphi(v_{i,n}),
  \end{array}
\right.
\end{equation}
where $\varphi(v)=\phi^k(v)$ and $\phi(v)=vg(v)$. The result follows
because each of these systems are under the hypothesis of
Proposition \ref{propoprob1attractor}.
\end{proof}

\begin{propo}\label{corolgk02}
Consider the equation  (\ref{gk}) being  $g$ a $\mathcal{C}^1$
function defined in an open set $\mathcal{U}\subseteq\R$. Let
$u_a\in\mathcal{U}$ be a hyperbolic attractor of $\phi(u)=ug(u)$.
Then, for all initial condition $(x_0,\ldots,x_{k-1})$ such that
$u_a\neq u_0=x_0\cdots x_{k-1}$ is in the basin of attraction of
$\{u=u_a\}$ we have that:
\begin{enumerate}
  \item[(a)] If $g(u_a)=1$, then the solution of (\ref{gk})
tends  to  a $k$-periodic orbit.
  \item[(b)] If $u_a=0$ and $g(0)=-1$, then the solution $\{x_n\}$ tends to a
$2k$-periodic orbit.
\end{enumerate}
In both cases, the period is not necessarily minimal.
\end{propo}

\begin{proof}
We consider again the systems (\ref{nousistemak30}) which, under the
current  hypotheses, also satisfy the ones of Corollary
\ref{corolaridiff}. Hence if $g(u_a)=1$, there exists $\ell_i(x_i)$
for $i=0,\ldots,k-1$ such that each sequence $\{z_{i,n}\}$ satisfies
$\lim\limits_{n\to\infty}z_{i,n}=\ell_i(x_i)$. If $g(0)=-1$, then
 there exist $\ell_{i,0}(x_i)$ and $\ell_{i,1}(x_i)$
such that $\lim\limits_{j\to\infty}z_{i,2j}=\ell_{i,0}(x_i)$ and
$\lim\limits_{j\to\infty}z_{i,2j+1}=\ell_{i,1}(x_i)$. It is easy to
see that for $i,j\in\{0,1\}$ each subsequence $\{z_{i,2k+j}\}$
corresponds with the subsequence $\{x_{4k+i+2j}\}$, obtaining that
there are $2k$  convergent subsequences of $\{x_n\}$.
\end{proof}

\subsection{Global dynamics of a  type of additive  difference
equations}\label{Aplications3}

In this section we consider the second order difference equations
\begin{equation}\label{s1}
x_{n+2}=-bx_{n+1}+g(x_{n+1}+bx_{n}).
\end{equation}  We call these equations
\emph{additive} ones. Equation (\ref{s1}) has the  associated map
$F(x,y)=(y,-by+g(y+bx))$ which preserves the fibration
$\mathcal{F}=\{y+bx=c,\,c\in\R\}$. So if $u=u^*$ is a fixed point of
$g$, the map preserves the fiber $y+bx=u^*$.

Setting $u_n=x_{n+1}+bx_{n}$, we get
\begin{equation}\label{s13}
\left\{
  \begin{array}{rl}
x_{n+1}&=u_n-bx_n,\\
u_{n+1}&=g(u_{n}),
  \end{array}
\right.
\end{equation}  which is a system of type (\ref{problema1})
with  $f_0(u)=u$; $f_1(u)\equiv -b$; and $\phi(u)=g(u)$.

It is easy to observe that if $|b|>1$, or $b=-1$ and $u^*\neq 0$
then there are iterates of map $F$ on the invariant fiber $y+bx=u^*$
which are unbounded, and therefore these cases are out of our
 scope. In fact it is straightforward to obtain the following
 result.

\begin{lem}
Consider the  equation  (\ref{s1}) where $g$ is a continuous
function  defined in an open set $\mathcal{U}\subseteq\R$. Let
$u^*\in\mathcal{U}$ be a fixed attracting point $u_{n+1}=g(u_n)$.
Then, for each initial condition $x_0,x_1$ such that $x_1+b
x_0=u^*$, we have
  $\lim\limits_{n\to\infty}x_{n}=u^*/(1+b)$ if $|b|<1$; the orbits
  are $2$-periodic if $b=1$ ;  are fixed points if
  $b=-1$ and $u^*=0$; and there are unbounded orbits if $b=-1$ and $u^*\neq 0$, or $|b|>1$.
\end{lem}

 For the rest of initial conditions, the
dynamics can be studied using Proposition \ref{propoprob1attractor}
and Corollary \ref{corolaridiff}, obtaining:

\begin{propo}
Consider the equation  (\ref{s1}) being
$g\in\mathcal{C}^1(\mathcal{U})$ function defined in an open set
$\mathcal{U}\subseteq\R$. Let $u^*\in\mathcal{U}$ be a hyperbolic
attracting point of $u_{n+1}=g(u_n)$. Then, for all initial
condition $x_0,x_1$ such that $u_0=x_1+bx_0$ is in the basin of
attraction of $u=u^*$, we have:
\begin{enumerate}
  \item[(a)]  If $|b|<1$, then
  $\lim\limits_{n\to\infty}x_{n}=u^*/(1+b)$.
  \item[(b)] If $b=1$ and system (\ref{s13}) has fast enough
convergence to the limit fiber $\{u=u_*\}$, then  $\{x_n\}$ tends to
a $2$-periodic orbit $\{\ell_0(x_0),\ell_1(x_1)\}$.
  \item[(c)]   If $b=-1$, $u^*=0$ and system (\ref{s13}) has fast enough
convergence to the limit fiber $\{u=0\}$, then there exists
$\ell(x_0)$   such that $\lim\limits_{n\to\infty}x_{n}=\ell(x_0)$.
\end{enumerate}
\end{propo}

\begin{proof} (a) If $|b|<1$, then system (\ref{s13}) is under the hypothesis of
Proposition \ref{propoprob1attractor}, hence
$$
\lim\limits_{n\to\infty}
x_{n}=\frac{f_0(u_*)}{1-f_1(u_*)}=\frac{u^*}{1+b}.
$$

To prove statements (b) and (c), observe that by applying Corollary
\ref{corolaridiff} to system (\ref{s13}), we obtain that if $b=1$,
then  for each $i=0,1$ there exists $\ell_i(x_i)$ such that
$\lim\limits_{k\to\infty} x_{2k+i}=\ell_i(x_i)$. And if $b=-1$ and
$u^*=0$, there exists $\ell(x_0)$ and $\ell_1$ such that
$\lim\limits_{n\to\infty}x_{n}=\ell(x_0)$.
\end{proof}

\medskip

 \noindent \textsl{Example J.}
Consider the equation
\begin{equation}\label{sumativak}
    x_{n+k}=x_n+f\left(\sum_{i=n}^{n+k-1} x_i\right).
\end{equation}
Adding the term $\sum_{i=n+1}^{n+k-1} x_i$ in  both sides; setting
$u_n=x_n+x_{n+1}+\cdots+x_{n+k-1}$, and after renaming
$z_{i,n+1}=x_{n+k}$ and $z_{i,n}=x_n$ we get  $k$ systems of type
(\ref{problema1}) associated to each initial condition $(x_i,u_i)$
for $i=0,\ldots,k-1$.
$$
\left\{
  \begin{array}{rl}
z_{i,n+1}&=f(u_n)+z_{i_n},\\
v_{i,n+1}&=\varphi(v_{i,n}),
  \end{array}
\right.
$$
where $\varphi(v)=\phi^k(v)$ and $\phi(v)=v+f(v)$. If $u_*$ is an
attractor of $u_{n+1}=u_n+f(u_n)$ and the system has the fast enough
convergence property to the limit fiber $\{u=u_*\}$. The dynamics on
the basin of attraction of this fiber can be studied using Theorem
\ref{propoprob1identitatv2}.

\medskip

\noindent \textsl{Example K.} In an analogous way as in all the
previous examples, by adding $ax_{n+1}$ in both sides of equation
$$
    x_{n+2}=ax_n+(1-a)x_{n+1}+f(x_{n+1}+ax_n),
$$
and setting $u_n=x_{n+1}+ax_n$, we get that the above equation can
be studied using via the system of type (\ref{problema1})
$$
\left\{
  \begin{array}{rl}
x_{n+1}&=u_n-ax_n,\\
u_{n+1}&=u_n+f(u_n).
  \end{array}
\right.
$$

\end{document}